\documentclass[11pt]{amsart}
\usepackage{amssymb,amsmath}
\usepackage{hyperref}
\usepackage{bm}
\usepackage[]{graphicx}
\usepackage{amssymb, epsfig}
\numberwithin{equation}{section}

\numberwithin{figure}{section}
\newcommand{\eps}{\varepsilon}

\newtheorem{theorem}{Theorem}[section]

\newtheorem{prop}[theorem]{Proposition}
\newtheorem{remark}[theorem]{Remark}
\begin{document}
%
%
%
%
\title[Dirichlet to Neumann  for  1-d Cubic  NLS]{On the Dirichlet to Neumann Problem for
the 1-dimensional Cubic  NLS equation on the Half-Line$^{\dag}$}
\author[D.~C. Antonopoulou]{D.~C. Antonopoulou$^{\$}$}
\thanks{$^{\dag}$ Research funded by ARISTEIA II grant no. 3964 from
the General Secretariat of Research and Technology, Greece}

\author[S. Kamvissis]{S. Kamvissis$^{*}$}
\thanks{$^*$ Department of Pure and Applied Mathematics,
University of Crete, GR--700 13 Heraklion, Greece, and, Institute
of Applied and Computational Mathematics, FORTH, GR--711 10
Heraklion, Greece, email: spyros@tem.uoc.gr}
\thanks{$^\$$ Department of Mathematics, University
of Chester, Thornton Science Park, CH2 4NU, UK, and, Institute of
Applied and Computational Mathematics, FORTH, GR--711 10
Heraklion, Greece, email: d.antonopoulou@chester.ac.uk}

%
%
\begin{abstract}
Initial-boundary value problems for 1-dimensional `completely
integrable' equations can be solved via an extension of the
inverse scattering method,  which is due to Fokas and his
collaborators. A crucial feature of this method is that it
requires the values of more  boundary data than given
for a well-posed problem. In the
case of cubic NLS, knowledge of the Dirichet data suffices to
make the problem well-posed but the Fokas method also requires
knowledge of the values of Neumann data. The study of the Dirichlet to
Neumann map is thus necessary before the application of the `Fokas
transform'. In this paper, we provide a rigorous study of this map
for a large class of decaying Dirichlet data. We show that the
Neumann data are also sufficiently decaying and that, hence, the Fokas
method can be applied.
%
\end{abstract}
%
%
\maketitle
\pagestyle{myheadings}
\thispagestyle{plain}
%
%
%
\section{Introduction}

In the last twenty  years there has been a series of results by
Fokas and collaborators on initial-boundary  value problems for
`completely integrable' equations. Their method
(introduced in  \cite{F} and further developed in \cite{F2},
\cite{F3}, \cite{FI}, \cite{FIS}, see also the book by Fokas
\cite{Fbook} for a comprehensive review)
generalises the `classical' theory of Kruskal et al., cf. for
example in \cite{Krusk}, which essentially reduces initial value
problems to Riemann-Hilbert problems via the scattering transform
(and then studies these problems by the method of inverse
scattering going back to Gelfand, Levitan and Marchenko). Instead
of the usual scattering transform the new transform of Fokas et
al. is based on the
simultaneous spectral analysis of both the x-problem and the
t-problem in the Lax pair.  So, initial-boundary  value problems
are also reduced  to Riemann-Hilbert factorisation
problems in the complex plane. The Fokas transform theory was
rigorously implemented to the NLS equation on the half-line with
Schwartz initial and boundary conditions in \cite{FIS}.

Alternatively, initial-boundary  value problems for PDEs can be studied via
PDE techniques, whose validity extends to non-integrable equations.
In this direction, we cite the seminal contributions of Kenig, Ponce, Vega
\cite{KPV}, \cite{KPV2} and Bourgain \cite{B} for initial value problems and
a (limited and perhaps random) selection of works by Carrol and Bu
\cite{CarBu},
Bona, Sun and Zhang \cite{BSZ}, Colliander  and Kenig \cite{ColK}, and Holmer
\cite{Hol} for initial-boundary  value problems.

One of the main advantages of the Riemann-Hilbert formulation  is
that one can use the  powerful nonlinear stationary phase and
steepest descent theory (see e.g. the book of P. Deift (\cite{D}) and \cite{KMM})
which gives rigorous results on the
asymptotic behavior of solutions to these  Riemann-Hilbert
problems (as some parameter goes to infinity) and hence it  extracts
asymptotics for the solution of the associated soliton equation
by use of the `Riemann-Hilbert method'. (See e.g. \cite{DZ} for the seminal paper on
long time asymptotics for mKdV; \cite{Kam}, \cite{KT}, \cite{KT2}
for the Toda lattice; \cite{DVZ} for KdV; \cite {DIZ} for NLS;
see also  \cite{K} and \cite{FK} for instances of
work that makes use of the Fokas theory.)

An important fact about the Fokas method
for well-posed initial-boundary value problems, is that
it involves unknown boundary values that should be characterised in terms of the given data.
The problem of extra boundary values has to be overcome one
way or another for the completely rigorous implementation of the
Fokas transform method. As far as finite times are concerned, this issue
is dealt with by the analysis of a 'global relation' in \cite{FIS}. If we are interested in infinite times
however, this has proved to be  harder. This is the problem we are solving  here,
at least in the defocusing  case.

In this paper, we consider the non-linear  Schr\"odinger equation
(NLS) with cubic non-linearity,  posed on the real
positive semi-axis $\mathbb{R}^+$
\begin{equation}\label{nls}
{\rm i}q_t+q_{xx}-2\lambda|q|^2q=0,\;\;\;\;x> 0,\;\;\;\;0<
t< +\infty,
\end{equation}
and initial-boundary data
\begin{equation}\label{ic}
\begin{split}
&q(x,0)=q_0(x),\;\;\;\;0\leq x<+\infty\\
&q(0,t)=Q(t),\;\;\;\;0\leq t<+\infty,\\
\end{split}
\end{equation}
where $q_0, Q$ are classical functions satisfying the
compatibility condition $q_0(0)=Q(0)$.

The case $\lambda =1$ is the defocusing case, while $\lambda =-1$ is the focusing case.

Back in 1991, Carrol and Bu in \cite{CarBu} established the existence of a
unique global classical solution $q\in C^1(L^2)\cap C^0(H^2)$ of
the problem \eqref{nls}-\eqref{ic}, with $q_0\in H^2$, $Q\in
C^2$ and $q_0(0)=Q(0)$, by using PDE theory.

On the other hand, it is well-known \cite{ZaSh} that the non-linear  Schr\"odinger equation
(NLS) with cubic non-linearity can be written as a Lax pair and that, at least the Cauchy problem
is `completely integrable', in other words it can be solved via the scattering transform.

Furthermore, in \cite{FIS} the authors used the Fokas
transform to solve the  problem  on the real
positive semi-axis, given values for the initial data and Dirichlet data
(which make the problem well-posed) and also the
Neumann data $P(t):=q_x(0,t)$. Actually what is required for that
theory to work is that the  Neumann data (as well as the
Dirichlet data) live in some class with nice decaying properties
such that the Fokas scattering transform can be properly defined.
This is exactly the problem we consider here: we will provide
several reasonably inclusive large classes of Dirichlet data,
such that both Dirichlet and Neumann data decay as
$t\rightarrow\infty$ fast enough for the scattering method to
work.

Obviously, our problem is a special case of the general one
treated in \cite{CarBu} and therefore, a global classical
solution $q$ uniquely exists.
But the Fokas transform method arrives at  a Riemann-Hilbert
formulation of the problem on the real
positive semi-axis, and hence long time asymptotics (as in \cite{FIS}) or
semiclassical asymptotics \cite{K} are also possible.

\subsection{Main results and Strategy} In Section 2, we consider
$q(0,t)$, $q_t(0,t)$ with mild polynomial decay as
$t\rightarrow\infty$ and prove that the Neumann data $q_x(0,t)$
is bounded in the $L^2(0,\infty)$-norm, while $q(x,t)$ is bounded
in the $H^1(0,\infty)$-norm in space uniformly for any $t\geq 0$;
these arguments are established for any real $\lambda$.

The defocusing NLS is further analyzed in Section 3. We first demonstrate that the
NLS solution decays to zero for large times. This yields
sharper estimates for the norm $\|q(\cdot,t)\|_{H^1(0,\infty)}$
which are crucial for the proof of the main theorem of this
Section, which states that  if the Dirichlet data $q(0,t)$,
$q_t(0,t)$ and $q_{tt}(0,t)$ have  sufficient
polynomial decay then $q_x(0,t)$ also have good decay and
thus the Fokas method is applicable.

For the focusing NLS, analyzed in
Section 4, the same result is derived under the assumption of decay of the
solution as $t\rightarrow\infty$, as well as a smallness assumption for
$\int_0^\infty|q(0,t)|^2dt$. Finally, in
Section 5 we present  our concluding remarks.

Our strategy for attacking the problem consists of the following steps:
\begin{enumerate}
\item
In Theorem \ref{thm1} we estimate
$$\int_0^\infty|q_x(0,t)|^2dt,$$
in terms of the Dirichlet data.
\item
We prove that if these data decay polynomially, then the
above norm is bounded, while $q$ is bounded in
$H^1(0,\infty)$-norm in space, uniformly for any $t\geq 0$.
\item
For the defocusing NLS, we establish a zero limiting profile of the
solution for large times, cf. Theorem \ref{thm2}.
\item We then proceed (using this decay) to  prove sharper estimates for
$\|q(\cdot,t)\|_{H^1(0,\infty)}$ in Theorem 3.4. This results in
an estimate of
$$\mathcal{A}:=\int_t^\infty|q_x(0,r)|^2dr,$$ in terms of the
Dirichlet data. In the proof of  Theorem 3.6. we
provides a bound for
$$\mathcal{B}:=\int_0^\infty|q_{xt}(0,t)|^2dt.$$
\item
Finally, we estimate $|q_x(0,t)|$ by using the bounds of
$\mathcal{A}$ and $\mathcal{B}$, to obtain the paper's main
result, which is Theorem 3.8.
\item
Considering the focusing NLS, the strategy is similar, but here we need
the assumption that the solution decays at large times, a fact
that we are able to prove only for the defocusing case.
\end{enumerate}


\section{General estimates}
In this section, we prove some basic  inequalities for the NLS
solution, for any real non-zero $\lambda$. Hence, these results are
applicable for both the defocusing and focusing case.

In what follows $(\cdot,\cdot)$ will denote the $L^2(0,\infty)$
inner product in space variables, while $\|\cdot\|$ the induced
norm. Furthermore, we shall use the symbol $\|\cdot\|_4$ for the
$L^4(0,\infty)$ norm in space and for any integer $p\geq 1$,
$\|\cdot\|_{L^p(0,t)}$ will denote the $L^p$ norm in the time
interval $(0,t)$.

The symbol $c$ will be used in general to denote  nonnegative
constants.

Throughout this paper  we make heavy use of the decay
of the solution $q$ of the NLS problem \eqref{nls}-\eqref{ic}
\begin{equation}\label{dc}
q(x,t)\rightarrow
0\;\;\mbox{as}\;\;\;\;x\rightarrow\infty\;\;\;\;\mbox{for
any}\;\;t\geq 0,
\end{equation}
as of course is guaranteed by the existence of a unique classical solution in \cite{CarBu}.
For simplicity, we also assume here that
\begin{equation}\label{ic0}
q(x,0)=q_0(x)=0\;\;\;\;\mbox{for any}\;\;x\geq 0.
\end{equation}
This simplifying assumption will be dropped in a forthcoming addendum to this paper.

The next Theorem presents a first general estimate in
$L^2(0,t)$-norm of the Neumann data $q_x(0,\cdot)$ valid for
$t\rightarrow\infty$ also. In fact, in order to obtain a uniform
estimate for any $t\geq 0$, we extent the proof of \cite{FIS}, cf.
Appendix D1, which was only established for $t$ taking values in a
bounded domain, thus involving  constants that may depend on its
upper bound.

\begin{theorem}\label{thm1} Let $q$ be the
unique global classical solution $q\in C^1(L^2)\cap C^0(H^2)$ of
the problem \eqref{nls}-\eqref{ic}, with  $Q\in
C^2$ and $Q(0)=0$,
under the 
assumption \eqref{ic0}. It holds that for any $t\geq 0$
\begin{equation}\label{mainineq1}
\begin{split}
\Big{(}\int_0^t|q_x(0,s)|^2ds\Big{)}^{1/2}=\|q_x(0,\cdot)\|_{L^2(0,t)}&\leq
c\|q(0,\cdot)\|_{L^2(0,t)}^{1/2}\|q_t(0,\cdot)\|_{L^2(0,t)}^{1/2}+c_1\|q(0,\cdot)\|_{L^4(0,t)}^2,
\end{split}
\end{equation}
for $c$ positive constant independent of $t$, $c_1\geq 0$
constant independent of $t$ with $c_1=0$ if $\lambda<0$ and
$c_1>0$ if $\lambda>0$.

For the case $\lambda<0$, we also impose the additional
assumption that $\|q(0,\cdot)\|_{L^2(0,t)}^2\leq c_2$ uniformly
in $t\geq 0$, for some $c_2>0$ sufficiently small.

The inequality \eqref{mainineq1} is also true when $t$ is
replaced by $\infty$.
\end{theorem}
\begin{remark}
We observe that in Appendix D1 of \cite{FIS} relation (D.10) is
proven for bounded time intervals $(0,T)$ and involving a general
continuous map of the $H^1(0,T)$ norm of $q(0,\cdot)=Q(t)$. In
this theorem we shall include the case $t\rightarrow \infty$,
while in view of \eqref{mainineq1}, we specify this map in terms
of specific powers of the appearing norms.
\end{remark}
\begin{proof}
Multiplying \eqref{nls} by $\bar{q}$ and integrating in
$x\in(0,\infty)$ we obtain after applying integration by parts
$${\rm
i}(q_t,q)-\|q_x\|^2-q_x(0,t)\bar{q}(0,t)-2\lambda(|q|^2q,q)=0.$$
Taking imaginary parts we arrive at
\begin{equation}\label{j**}
\frac{d}{dt}\|q\|^2=2{\rm Im}\{q_x(0,t)\bar{q}(0,t)\},
\end{equation}
 which after
integration in time gives
\begin{equation}\label{mt2}
\|q(\cdot,t)\|^2\leq
2\|q_x(0,\cdot)\|_{L^2(0,t)}\|q(0,\cdot)\|_{L^2(0,t)}.
\end{equation}

Multiplying now \eqref{nls} by $\bar{q_t}$ and integrating in
$x\in(0,\infty)$ we obtain
$${\rm i}
\|q_t\|^2-(q_x,q_{xt})-q_x(0,t)\bar{q_t}(0,t)-2\lambda(|q|^2q,q_t)=0,$$
while taking real parts and since
$$\frac{d}{dt}[|q|^2|q|^2]=4{\rm Re}\{|q|^2q\bar{q_t}\},$$ we have
\begin{equation}\label{j0}
\frac{d}{dt}\|q_x\|^2=-\lambda\frac{d}{dt}\||q|^2\|^2-2{\rm
Re}\{q_x(0,t)\bar{q_t}(0,t)\}.
\end{equation}
 Integration in time yields
$$\|q_x(\cdot,t)\|^2=-\lambda\||q(\cdot,t)|^2\|^2-2\int_0^t{\rm Re}\{q_x(0,s)\bar{q_t}(0,s)\}ds,$$
and thus,
\begin{equation}\label{mt3}
\|q_x(\cdot,t)\|^2+\lambda\|q(\cdot,t)\|_{L^4(0,\infty)}^4\leq
2\|q_x(0,\cdot)\|_{L^2(0,t)}\|q_t(0,\cdot)\|_{L^2(0,t)}.
\end{equation}

We now multiply \eqref{nls} with $\bar{q_x}$ integrate in space
and take real parts which yields
$$-{\rm Im}(q_t,q_x)+{\rm Re}(q_{xx},q_x)-2\lambda{\rm
Re}(|q|^2q,q_x)=0.$$ Using that
$${\rm Im}(q_t,q_x)=-\frac{1}{2}{\rm
i}\frac{d}{dt}(q,q_x)-\frac{1}{2}{\rm i}q(0,t)\bar{q_t}(0,t),$$
and the relations
$${\rm Re}(q_{xx},q_x)=-\frac{1}{2}|q_x(0,t)|^2,\;\;\;\;{\rm
Re}(|q|^2q,q_x)=-\frac{1}{4}|q(0,t)|^4,$$ we obtain
\begin{equation}\label{ja1}
|q_x(0,t)|^2={\rm i}\frac{d}{dt}(q,q_x)+{\rm
i}q(0,t)\bar{q_t}(0,t)+\lambda|q(0,t)|^4.
\end{equation}
 Integrating the above
in time, we get
$$\int_0^t|q_x(0,s)|^2ds={\rm i}(q(\cdot,t),q_x(\cdot,t))+{\rm
i}\int_0^tq(0,s)\bar{q_t}(0,s)ds+\lambda\int_0^t|q(0,s)|^4ds,$$
and so,
\begin{equation}\label{mt4}
\|q_x(0,\cdot)\|_{L^2(0,t)}^2\leq\|q(\cdot,t)\|\|q_x(\cdot,t)\|+
\|q(0,\cdot)\|_{L^2(0,t)}\|q_t(0,\cdot)\|_{L^2(0,t)}+\lambda\|q(0,\cdot)\|_{L^4(0,t)}^4.
\end{equation}
We shall estimate $\|q(\cdot,t)\|\|q_x(\cdot,t)\|$. Indeed, by
\eqref{mt2} and \eqref{mt3} and the Sobolev inequality
$$\|q(\cdot,t)\|_{L^4(0,\infty)}^4\leq \|q(\cdot,t)\|^3\|q_x(\cdot,t)\|,$$
we arrive at
\begin{equation*}
\begin{split}
\|q(\cdot,t)\|\|q_x(\cdot,t)\|\leq &
c\|q_x(0,\cdot)\|_{L^2(0,t)}\|q(0,\cdot)\|_{L^2(0,t)}^{1/2}\|q_t(0,\cdot)\|_{L^2(0,t)}^{1/2}\\
&+\tilde{c}\|q_x(0,\cdot)\|_{L^2(0,t)}^{1/2}\|q(0,\cdot)\|_{L^2(0,t)}^{1/2}\|q(\cdot,t)\|^{3/2}\|q_x(\cdot,t)\|^{1/2}\\
\leq &c\|q_x(0,\cdot)\|_{L^2(0,t)}\|q(0,\cdot)\|_{L^2(0,t)}^{1/2}\|q_t(0,\cdot)\|_{L^2(0,t)}^{1/2}\\
&+\tilde{c}c\|q_x(0,\cdot)\|_{L^2(0,t)}\|q(0,\cdot)\|_{L^2(0,t)}\|q(\cdot,t)\|^2+
\tilde{c}c_0\|q(\cdot,t)\|\|q_x(\cdot,t)\|,
\end{split}
\end{equation*}
where $\tilde{c}=0$ if $\lambda=1$ and $c_0$ can be made as small as we want.
Hence, using the above and the estimate \eqref{mt2} for $\|q\|$,
we have
\begin{equation}\label{mt5}
\begin{split}
\|q(\cdot,t)\|\|q_x(\cdot,t)\|\leq
&c\|q_x(0,\cdot)\|_{L^2(0,t)}\|q(0,\cdot)\|_{L^2(0,t)}^{1/2}
\|q_t(0,\cdot)\|_{L^2(0,t)}^{1/2}\\
&+\tilde{c}c\|q_x(0,\cdot)\|_{L^2(0,t)}\|q(0,\cdot)\|_{L^2(0,t)}\|q(\cdot,t)\|^2\\
\leq
&c\|q_x(0,\cdot)\|_{L^2(0,t)}\|q(0,\cdot)\|_{L^2(0,t)}^{1/2}\|q_t(0,\cdot)\|_{L^2(0,t)}^{1/2}
+\tilde{c}c\|q_x(0,\cdot)\|_{L^2(0,t)}^2\|q(0,\cdot)\|_{L^2(0,t)}^2\\
\leq &
\hat{c}_0\|q_x(0,\cdot)\|_{L^2(0,t)}^2+c\|q(0,\cdot)\|_{L^2(0,t)}\|q_t(0,\cdot)\|_{L^2(0,t)}\\
&+\tilde{c}c\|q_x(0,\cdot)\|_{L^2(0,t)}^2\|q(0,\cdot)\|_{L^2(0,t)}^2,
\end{split}
\end{equation}
where again $\hat{c}_0$ can be made as small as we want.
Using now \eqref{mt5} in \eqref{mt4} we obtain
\begin{equation*}
\begin{split}
\|q_x(0,\cdot)\|_{L^2(0,t)}^2\leq&\hat{c}_0\|q_x(0,\cdot)\|_{L^2(0,t)}^2
+c\|q(0,\cdot)\|_{L^2(0,t)}\|q_t(0,\cdot)\|_{L^2(0,t)}\\
&+\tilde{c}c\|q_x(0,\cdot)\|_{L^2(0,t)}^2\|q(0,\cdot)\|_{L^2(0,t)}^2+
\|q(0,\cdot)\|_{L^2(0,t)}\|q_t(0,\cdot)\|_{L^2(0,t)}+\lambda\|q(0,\cdot)\|_{L^4(0,t)}^4,
\end{split}
\end{equation*}
and thus we  obtain \eqref{mainineq1}, i.e.
\begin{equation*}
\begin{split}
\|q_x(0,\cdot)\|_{L^2(0,t)}^2\leq
c\|q(0,\cdot)\|_{L^2(0,t)}\|q_t(0,\cdot)\|_{L^2(0,t)}+c_1\|q(0,\cdot)\|_{L^4(0,t)}^4,
\end{split}
\end{equation*}
where $c_1=c_1(\lambda)=0$ if $\lambda=-1$, because
we had made $\hat{c}_0$ appropriately small; when $\lambda=-1$,
i.e. when $\tilde{c}\neq 0$ we use  the assumption that
$\|q(0,\cdot)\|_{L^2(0,t)}^2$ is small uniformly in $t$.

\end{proof}

\begin{remark}\label{rem1}Under the assumptions of the previous Theorem \ref{thm1} and
due to  estimate \eqref{mainineq1}, we note that if $q(0,t)$,
$q_t(0,t)$, have polynomial decay of order
$\mathcal{O}(t^{-\alpha})$, $\mathcal{O}(t^{-\beta})$ respectively
as $ t \to \infty,$ with $\alpha,\beta>1/2$ then:\\
\\
\\
1) It holds that
\begin{equation}\label{unifqx}\int_0^\infty
|q_x(0,t)|^2dt<\infty.
\end{equation}
\\
\\
\\
2) Furthermore for $\alpha$, $\beta$ as in 1), the estimate
\eqref{unifqx}, together with \eqref{mt2} and \eqref{mt3} yield
that there exists $c>0$ independent of $t$, such that for any
$t\geq 0$
\begin{equation}\label{j***}
\|q(\cdot,t)\|\leq c,
\end{equation}
and
\begin{equation}\label{j***1}
\|q_x(\cdot,t)\|\leq c.
\end{equation}
(Note that in order to prove the inequality \eqref{j***1}, and
since an $L^4$-norm term is present in \eqref{mt3}, we use
additionally the Sobolev Embedding Theorem
$$\|q\|_4^4\leq \|q\|^3\|q_x\|\leq c\|q\|^6+c_0\|q_x\|^2,$$
where $c_0>0$ can be made as small as we want.)

For 
$\alpha,\beta$ sufficiently large, we shall eventually prove
sharper estimates for these norms.\\
\\
\\
3) In addition, we have that there exists $c>0$ independent of
$t$ such that for any $t\geq 0$
\begin{equation}\label{limq}
|q(x,t)|\leq c\|q(\cdot,t)\|^{1/2}\|q_x(\cdot,t)\|^{1/2}\leq c,
\end{equation}
and so, $$\displaystyle{\lim_{t \to \infty}}|q(x,t)|\leq c,$$ for
any $x$, if the limit exists.
\end{remark}

\section{Defocusing NLS}\label{def}
We now proceed by considering the case $\lambda=1$, i.e. the
defocusing NLS equation.

Our first aim is to establish, under certain assumptions for the
initial data, that the solution of \eqref{nls}-\eqref{ic} when
$\lambda=1$ decays to $0$ for any $x$ as $t\rightarrow \infty$.
This is achieved by proving that the $L^4$ norm of the solution
decays in time.

More specifically, we prove the next main theorem.
\begin{theorem}\label{thm2} Let $q$ be the
unique global classical solution $q\in C^1(L^2)\cap C^0(H^2)$ of
the problem \eqref{nls}-\eqref{ic}, with  $Q\in
C^2$ and $Q(0)=0$,
with $\lambda=1$,
under the assumption  \eqref{ic0}. Furthermore,
assume that as $t\rightarrow\infty$
$$q(0,t)= \mathcal{O}(t^{-\alpha}),\;\;\;\;q_t(0,t)=
\mathcal{O}(t^{-\beta}),\;\;\;\;\mbox{for}\;\;\alpha>
3/2\;\;\;\mbox{and}\;\;\beta> 5/2.$$ It holds that there exists a
positive constant $c$ independent of $t$ such that
\begin{equation}\label{j2}
\int_0^\infty|q(x,t)|^4dx:=\|q(\cdot,t)\|_4^4\leq
\frac{c}{t}\;\;\;\;\mbox{for any}\;\;t\geq 1.
\end{equation}
\end{theorem}
\begin{proof}
Again let the Dirichlet data be $$Q(t):=q(0,t),$$ and
let us denote the Neumann data by $$P(t):=q_x(0,t).$$

We set
$$A(t):=\int_0^\infty \Big{(}|q_x(x,t)|^2+|q(x,t)|^4\Big{)}dx.$$ Then by
\eqref{j0} we have
\begin{equation}\label{j4}
\partial_t A(t)=-2{\rm Re}\{q_x(0,t)\bar{q_t}(0,t)\}=-2{\rm
Re}\{P(t)\bar{Q}(t)\}.
\end{equation}

We now define
$$y(t)={\rm Im}\int_0^\infty x\bar{q}(x,t)q_x(x,t)dx,$$
to obtain
\begin{equation}\label{j5}
\begin{split}
y_t={\rm Im}\int_0^\infty(x\bar{q_t}q_x+x\bar{q}q_{xt})dx=& {\rm
Im}\int_0^\infty(x\bar{q_t}q_x-\bar{q}q_t-x\bar{q_x}q_t)dx+{\rm{Im}}[x\bar{q}q_t]_{x=0}^{\infty}\\
=&-{\rm Im}\int_0^\infty\bar{q}q_t dx+2{\rm Im}\int_0^\infty
x\bar{q_t}q_xdx.
\end{split}
\end{equation}
Using the NLS equation \eqref{nls}, we have
$${\rm i}(xq_t,q_x)+(xq_{xx},q_x)-2(x|q|^2q,q_x)=0,$$
which gives, by taking real parts,
\begin{equation}\label{j6}
-{\rm Im}(xq_t,q_x)+{\rm Re}(xq_{xx},q_x)-2{\rm
Re}(x|q|^2q,q_x)=0.
\end{equation}
We observe that
$$(xq_{xx},q_x)=-(q_x,q_x)-(xq_x,q_{xx})+[x|q_x|^2]_{x=0}^{\infty},$$
and thus,
\begin{equation}\label{j7}
{\rm Re}(xq_{xx},q_x)=-\frac{1}{2}\|q_x\|^2.
\end{equation}
Furthermore,
$$(x|q|^2q,q_x)=-\int_0^\infty|q|^4dx-\int_0^\infty
x[2q_x|q|^2\bar{q}+q\bar{q_x}|q|^2]dx,$$ so, we obtain
\begin{equation}\label{j8}
{\rm Re}(x|q|^2q,q_x)=-\frac{1}{4}\int_0^\infty|q|^4dx.
\end{equation}
Relations \eqref{j6}, \eqref{j7}, \eqref{j8}, yield
\begin{equation}\label{j9}
0=-{\rm
Im}(xq_t,q_x)-\frac{1}{2}\|q_x\|^2+\frac{1}{2}\int_0^\infty|q|^4dx.
\end{equation}
In addition, the NLS equation \eqref{nls} gives
$${\rm i}(q_t,q)+(q_{xx},q)-2(|q|^2q,q)=0,$$
thus, taking real parts, we obtain
\begin{equation}\label{j10}
-{\rm Im}(q_t,q)+{\rm Re}(q_{xx},q)-2{\rm Re}(|q|^2q,q)=0,
\end{equation}
while
\begin{equation}\label{j11}
(q_{xx},q)=-\|q_x\|^2+[q_x\bar{q}]_{x=0}^\infty=-\|q_x\|^2-q_x(0,t)\bar{q}(0,t).
\end{equation}
By \eqref{j10}, \eqref{j11} we have
\begin{equation}\label{j12}
0=-{\rm Im}(q_t,q)-\|q_x\|^2-{\rm
Re}\{q_x(0,t)\bar{q}(0,t)\}-2\int_0^\infty |q|^4dx.
\end{equation}
Therefore, \eqref{j5}, together with \eqref{j12}, and \eqref{j9}
give
\begin{equation}\label{j13}
\begin{split}
y_t(t)=&2\|q_x\|^2+{\rm
Re}\{q_x(0,t)\bar{q}(0,t)\}+\int_0^\infty|q|^4dx\\
=&{\rm Re}P(t)\bar{Q}(t)+A(t)+\int_0^\infty |q_x(x,t)|^2dx.
\end{split}
\end{equation}

Observe now that
\begin{equation}\label{j14}
\partial_t\int_0^\infty x^2|q|^2dx=\int_0^\infty
x^2[q_t\bar{q}+q\bar{q_t}]dx=2{\rm Re}\int_0^\infty
x^2q_t\bar{q}dx.
\end{equation}
Multiplying the NLS equation \eqref{nls} by $x^2$ we obtain
\begin{equation*}
\begin{split}
0=&{\rm i}(x^2q_t,q)+(x^2q_{xx},q)-2(x^2|q|^2q,q)\\
=&{\rm
i}(x^2q_t,q)-(2xq_x,q)-(x^2q_x,q_x)+[x^2q_x\bar{q}]_{x=0}^\infty-2\int_0^\infty
x^2|q|^4dx\\
=&{\rm i}(x^2q_t,q)-(2xq_x,q)-(x^2q_x,q_x)-2\int_0^\infty
x^2|q|^4dx.
\end{split}
\end{equation*}
So, taking imaginary parts yields
$$
2{\rm Re}(x^2q_t,q)=2{\rm Im}(2xq_x,q)=4y,
$$
which gives by \eqref{j14}
\begin{equation}\label{j15}
\partial_t\int_0^\infty x^2|q|^2dx=4y.
\end{equation}

But, we have
$$4ty_t=\partial_t(4ty)-4y,$$
and thus, by \eqref{j13} and \eqref{j15}, we obtain
$$4t{\rm Re}\{P\bar{Q}\}+4tA+4t\int_0^\infty
|q_x|^2dx=\partial_t(4ty)-\partial_t\int_0^\infty x^2|q|^2dx,$$
and equivalently, replacing $A$
\begin{equation}\label{j16}
\begin{split}
\partial_t(4ty)=&4t{\rm
Re}\{P\bar{Q}\}+4tA+4t\int_0^\infty|q_x|^2dx+\partial_t\int_0^\infty
x^2|q|^2dx\\
=&4t{\rm
Re}\{P\bar{Q}\}+8t\int_0^\infty|q_x|^2dx+4t\int_0^\infty
|q|^4dx+\partial_t\int_0^\infty x^2|q|^2dx\\
=&4t{\rm
Re}\{P\bar{Q}\}+\partial_t\Big{(}4t^2\int_0^\infty|q_x|^2dx\Big{)}-4t^2\int_0^\infty
\partial_t|q_x|^2dx+4t\int_0^\infty |q|^4dx+\partial_t\int_0^\infty x^2|q|^2dx.
\end{split}
\end{equation}
So, using \eqref{j16} we arrive at
\begin{equation}\label{j17}
\begin{split}
&\partial_t\Big{[}4ty-\int_0^\infty
x^2|q|^2dx-4t^2\int_0^\infty|q_x|^2dx\Big{]}=\\
&4t{\rm
Re}\{P\bar{Q}\}-4t^2\int_0^\infty\partial_t|q_x|^2dx+4t\int_0^\infty|q|^4dx.
\end{split}
\end{equation}
Using \eqref{j0} in \eqref{j17}, we obtain
\begin{equation}\label{j18}
\begin{split}
&\partial_t\Big{[}4ty-\int_0^\infty
x^2|q|^2dx-4t^2\int_0^\infty|q_x|^2dx\Big{]}=\\
&4t{\rm
Re}\{P\bar{Q}\}+4t^2\partial_t\Big{(}\int_0^\infty|q|^4dx\Big{)}+8t^2{\rm
Re}\{P\bar{Q_t}\}+4t\int_0^\infty|q|^4dx.
\end{split}
\end{equation}
But,
\begin{equation*}
\begin{split}
4t^2\partial_t\Big{(}\int_0^\infty
|q|^4dx\Big{)}+4t\int_0^\infty|q|^4dx=&\partial_t\Big{(}4t^2\int_0^\infty|q|^4dx\Big{)}
-4t\int_0^\infty|q|^4dx,
\end{split}
\end{equation*}
and therefore, \eqref{j18} yields
\begin{equation*}
\begin{split}
&\partial_t\Big{[}4ty-\int_0^\infty x^2|q|^2dx-4t^2\int_0^\infty
|q_x|^2dx-4t^2\int_0^\infty|q|^4dx\Big{]}\\
&=4t{\rm Re}\{P\bar{Q}\}+8t^2{\rm
Re}\{P\bar{Q_t}\}-4t\int_0^\infty|q|^4dx,
\end{split}
\end{equation*}
and so,
\begin{equation*}
\begin{split}
\partial_t\Big{(}t^2\int_0^\infty|q|^4dx\Big{)}=&\frac{1}{4}\partial_t\Big{[}4ty-\int_0^\infty
x^2|q|^2dx-4t^2\int_0^\infty|q_x|^2dx\Big{]}\\
&-t{\rm
Re}\{P\bar{Q}\}-2t^2{\rm
Re}\{P\bar{Q_t}\}+t\int_0^\infty|q|^4dx.
\end{split}
\end{equation*}
Integrating the above in time in the interval $(0,t)$, we get
\begin{equation}\label{j19}
\begin{split}
t^2\|q(\cdot,t)\|_4^4=&\frac{1}{4}\Big{[}4ty-\int_0^\infty
x^2|q(x,t)|^2dx-4t^2\int_0^\infty|q_x(x,t)|^2dx\Big{]}\\
&-\int_0^tr{\rm Re}\{P(r)\bar{Q}(r)\}dr-
2\int_0^tr^2{\rm
Re}\{P(r)\bar{Q_r}(r)\}dr+\int_0^tr\|q(\cdot,r)\|_4^4dr.
\end{split}
\end{equation}
Observe now that
\begin{equation}\label{j20}
4ty=\int_0^\infty[x^2|q|^2+4t^2|q_x|^2-|xq+2{\rm i}tq_x|^2]dx.
\end{equation}
While by \eqref{j19} we have
\begin{equation}\label{j21}
\begin{split}
t^2\|q(\cdot,t)\|_4^4=\frac{1}{4}\Big{[}&\int_0^\infty(x^2|q(x,t)|^2+4t^2|q_x(x,t)|^2-|xq+2{\rm
i}tq_x(x,t)|^2-x^2|q(x,t)|^2)dx\\
&-4t^2\int_0^\infty|q_x(x,t)|^2dx\Big{]}\\
-&\int_0^tr{\rm
Re}\{P(r)\bar{Q}(r)\}dr-2\int_0^tr^2{\rm
Re}\{P(r)\bar{Q_r}(r)\}dr+\int_0^tr\|q(\cdot,r)\|_4^4dr\\
=&\frac{1}{4}\Big{[}-\int_0^\infty|xq(x,t)+2{\rm
i}tq_x(x,t)|^2dx\Big{]}\\
&-\int_0^t r{\rm
Re}\{P(r)\bar{Q}(r)\}dr-2\int_0^tr^2{\rm
Re}\{P(r)\bar{Q_r}(r)\}dr+\int_0^tr\|q(\cdot,r)|_4^4dr.
\end{split}
\end{equation}
Hence, \eqref{j21} gives
\begin{equation}\label{j22}
t^2\|q(\cdot,t)\|_4^4\leq \int_0^t
r\|q(\cdot,r)\|_4^4dr+F(P,Q,Q_t,t),
\end{equation}
for
$$F(P,Q,Q_t,t):=-\int_0^t r{\rm
Re}\{P(r)\bar{Q}(r)\}dr-2\int_0^t r^2{\rm
Re}\{P(r)\bar{Q_r}(r)\}dr.$$

Next, we prove that there exists $c>0$ independent of
$t$ such that for any $t\geq 0$,
\begin{equation}\label{j25}
|F(P,Q,Q_t,t)|\leq c.
\end{equation}
Indeed,  if $a>3/2$ and $\beta>5/2$, then we have for
$F=F(P(t),Q(t),Q_t(t),t)$
\begin{equation*}
\begin{split}
|F|\leq &c\Big{(}\int_0^t|P(r)|^2dr\Big{)}^{1/2}\Big{(}\int_0^t
r^2|Q(r)|^2dr\Big{)}^{1/2}+c\Big{(}\int_0^t|P(r)|^2dr\Big{)}^{1/2}\Big{(}\int_0^tr^4|Q_r(r)|^2dr\Big{)}^{1/2}\\
\leq & c\Big{(}\int_0^t|P(r)|^2dr\Big{)}^{1/2}\leq c,
\end{split}
\end{equation*}
since \eqref{unifqx} is true. Thus, \eqref{j25} follows.

Using now \eqref{j25} at \eqref{j22}, we obtain that there exists
$c>0$ independent of $t$ such that
\begin{equation}\label{j26}
t^2\|q(\cdot,t)\|_4^4\leq \int_0^t r\|q(\cdot,r)\|_4^4dr+c,
\end{equation}
and therefore, using \eqref{mt3} for $\lambda>0$, we have for any
$t\geq 1$
\begin{equation*}
\begin{split}
t^2\|q(\cdot,t)\|_4^4\leq &\int_0^1 r\|q(\cdot,r)\|_4^4dr+\int_1^t
r\|q(\cdot,r)\|_4^4dr+c\\
\leq & \int_0^1 \|q(\cdot,r)\|_4^4dr+\int_1^t
r\|q(\cdot,r)\|_4^4dr+c\\
\leq &
c\Big{(}\int_0^\infty|q_x(0,r)|^2dr\Big{)}^{1/2}\Big{(}\int_0^\infty|q_t(0,r)|^2dr\Big{)}^{1/2}+\int_1^t
r\|q(\cdot,r)\|_4^4dr+c\\
\leq &\int_1^t r\|q(\cdot,r)\|_4^4dr+c,
\end{split}
\end{equation*}
where we used \eqref{unifqx} and the fact that $\beta>1/2$. So,
there exists $c>0$ independent of $t$ such that for any $t\geq 1$
\begin{equation*}
\begin{split}
t^2\|q(\cdot,t)\|_4^4\leq \int_1^t r\|q(\cdot,r)\|_4^4dr+c.
\end{split}
\end{equation*}
Applying  Gronwall's lemma to the above, we obtain \eqref{j2}.
\end{proof}
\begin{remark}\label{remn1}
Similar relations to \eqref{j2} have been proved in
\cite{Bar,CarBu} for the defocusing NLS; in \cite{Bar} for zero
Dirichlet data $q(0,t)$, while in \cite{CarBu} under additional
assumptions.

But here we have provided the first  sufficient polynomial decay conditions
as $t\rightarrow\infty$ involving not only $q(0,t)=Q(t)$ but
$q_t(0,t)=Q_t(t)$ as well, which make  the $L^4$-norm decay true.
This permits us to establish a zero limiting behaviour for the
solution, uniformly in space, as $t\rightarrow\infty$, which is
crucial for proving the main results of this paper, i.e. Theorems 3.8 and 4.4.
\end{remark}
A direct corollary  of the previous theorem is presented in the next
proposition.

\begin{prop}\label{prop1}
Under the assumptions of Theorem \ref{thm2} it holds that the
pointwise limits as $t \to \infty$ are
\begin{equation}\label{j27}
\displaystyle{\lim_{t\rightarrow\infty}}q(x,t)=0,\;\;\;\;
\displaystyle{\lim_{t\rightarrow\infty}}(\partial_x^kq)(x,t)=0,\;\;\;\;
\displaystyle{\lim_{t\rightarrow\infty}}q_t(x,t)=0,\;\;\;\;
\displaystyle{\lim_{t\rightarrow\infty}}(\partial_x^k\partial_t
q)(x,t)=0,
\end{equation}
for any $x>0$, and any integer $k\geq 1$.
\end{prop}
\begin{proof}
Indeed, we have by Theorem \ref{thm2}
$$\displaystyle{\lim_{t\rightarrow\infty}}\|q(\cdot,t)\|_4\leq
\displaystyle{\lim_{t\rightarrow\infty}}\frac{c}{t^{1/4}}=0,$$
which implies the first two relations of \eqref{j27}.
In fact,
observing that
$$\frac{d}{dx}q^3(x,t)=3q^2(x,t)q_x(x,t),$$
we have
\begin{equation}\label{apo1}
\begin{split}
|q^3(x,t)|=&|q^3(0,t)+\int_0^x3q^2(y,t)q_x(y,t)dy|\\
\leq &|q(0,t)|^3+c\Big{(}\int_0^\infty
|q(y,t)|^4dy\Big{)}^{1/2}\Big{(}\int_0^\infty
|q_x(y,t)|^2dy\Big{)}^{1/2}.
\end{split}
\end{equation}
So, under the assumptions of decaying initial data
$q(0,t)\rightarrow 0$ as $t\rightarrow\infty$, and  by Remark
\ref{rem1}, if $q(0,t)$, $q_t(0,t)$ have polynomial decay of
order $\mathcal{O}(t^{-\alpha})$, $\mathcal{O}(t^{-\beta})$
respectively as $ t \to \infty,$ with $\alpha,\beta>1/2$, we have
(cf. \eqref{j***1})
$$
\|q_x(\cdot,t)\|\leq c,
$$
and thus, \eqref{apo1} yields
\begin{equation}\label{apo2}
\begin{split}
\displaystyle{\lim_{t\rightarrow\infty}}|q^3(x,t)|
 \leq
c\displaystyle{\lim_{t\rightarrow\infty}}\|q(\cdot,t)\|_4^{2}=0.
\end{split}
\end{equation}

The relations involving the derivative in time
follow from the NLS equation \eqref{nls} (differentiating it in $x$
when needed) by using the fact that $q(x,\infty)=0$, and
$(\partial_x^kq)(x,\infty)=0$.

Here, we note that by Remark \ref{rem1}, we have already proved
that $q(x,t)$ is bounded uniformly in $x,\;t>0$.
\end{proof}

Now, we can prove sharper estimates than these of Remark
\ref{rem1} for $q$ and $q_x$ in the $L^2(0,\infty)$ norm in
space. This is achieved by using the fact that the solution decays as
$t\rightarrow\infty$ and the polynomial decay of the Dirichlet data
for large times.
\begin{theorem}\label{thm3}
Under the assumptions of Theorem \ref{thm2} it holds that there
exists $c>0$ independent of $t$ such that for any $t\geq 0$
\begin{equation}\label{j28}
\|q(\cdot,t)\|^2\leq
c\Big{(}\int_t^\infty|q(0,r)|^2dr\Big{)}^{1/2},
\end{equation}
and
\begin{equation}\label{j29}
\|q_x(\cdot,t)\|^2\leq
c\Big{(}\int_t^\infty|q_t(0,r)|^2dr\Big{)}^{1/2}.
\end{equation}
\end{theorem}
\begin{proof}
Integrating \eqref{j**} in $(t,\infty)$ and using that by
Proposition \ref{prop1} $\|q(\cdot,\infty)\|=0$, we obtain
$$\|q(\cdot,t)\|^2=-2\int_t^\infty{\rm Im}\{q_x(0,r)\bar{q}(0,r)\}dr,$$
which yields for any $t\geq 0$
\begin{equation*}
\begin{split}
\|q(\cdot,t)\|^2\leq&
c\Big{(}\int_0^\infty|q_x(0,r)|^2dr\Big{)}^{1/2}\Big{(}\int_t^\infty|q(0,r)|^2dr\Big{)}^{1/2}\\
\leq &c\Big{(}\int_t^\infty|q(0,r)|^2dr\Big{)}^{1/2},
\end{split}
\end{equation*}
where we used \eqref{unifqx}. Here, $c>0$ is a constant independent of $t$.
Thus, \eqref{j28} holds true.

Furthermore, \eqref{j0} when integrated in $(t,\infty)$ gives
\begin{equation*}
\begin{split}
\|q_x(\cdot,t)\|^2=2\int_t^\infty{\rm Re}\{
q_x(0,r)\bar{q_t}(0,r)\}dr-\|q(\cdot,t)\|_4^4,
\end{split}
\end{equation*}
where we used that by Proposition \ref{prop1}
$\|q(\cdot,\infty)\|=0$ and $\|q_x(\cdot,\infty)\|=0$. Thus, for
any $t\geq 0$ we have
\begin{equation*}
\begin{split}
\|q_x(\cdot,t)\|^2\leq c\Big{(}\int_t^\infty
|q_t(0,r)|^2dr\Big{)}^{1/2},
\end{split}
\end{equation*}
since \eqref{unifqx} is true. Here again, $c>0$ is  a constant independent of $t$.
So, \eqref{j29} follows.
\end{proof}

The next proposition presents an estimate for the Neumann data in
$L^2(t,\infty)$ in terms of the Dirichlet data and its derivative
on $(t,\infty)$. This result, similarly to the previous
theorem connects the aforementioned quantities and is useful when
large times  decay  for $q(0,t)$ and $q_t(0,t)$ is taken into account.
\begin{prop}\label{prop2}
Under the assumptions of Theorem \ref{thm2} it holds that there
exists $c>0$ independent of $t$ such that for any $t\geq 0$
\begin{equation}\label{j30}
\begin{split}
\int_t^\infty |q_x(0,r)|^2dr\leq &
c\Big{(}\int_t^\infty|q(0,r)|^2dr\Big{)}^{1/4}\Big{(}\int_t^\infty|q_t(0,r)|^2dr\Big{)}^{1/4}\\
&+c\Big{(}\int_t^\infty|q(0,r)|^2dr\Big{)}^{1/2}\Big{(}\int_t^\infty|q_t(0,r)|^2dr\Big{)}^{1/2}
+\int_t^\infty|q(0,r)|^4dr.
\end{split}
\end{equation}
\end{prop}
\begin{proof}
We integrate \eqref{ja1} in time in $(t,\infty)$ and use that
$q(x,\infty)=0$ for any $x\in (0,\infty)$, to arrive at
\begin{equation*}
\begin{split}
\int_t^\infty|q_x(0,r)|^2dr=&-{\rm
i}(q(\cdot,t),q_x(\cdot,t))+{\rm i}\int_t^\infty
q(0,r)\bar{q_t}(0,r)dr+\int_t^\infty|q(0,r)|^4dr\\
\leq&
c\Big{(}\int_t^\infty|q(0,r)|^2dr\Big{)}^{1/4}\Big{(}\int_t^\infty|q_t(0,r)|^2dr\Big{)}^{1/4}\\
&+c\Big{(}\int_t^\infty|q(0,r)|^2dr\Big{)}^{1/2}\Big{(}\int_t^\infty|q_t(0,r)|^2dr\Big{)}^{1/2}
+\int_t^\infty|q(0,r)|^4dr,
\end{split}
\end{equation*}
where we applied the estimates \eqref{j28} and \eqref{j29}.
\end{proof}

Now we have the following  theorem.

\begin{theorem}\label{thm3}
Under the assumptions of Theorem \ref{thm2}, for
$\alpha>5/2,\;\;\beta>5/2$, and if as $t\rightarrow\infty$
$$q_{tt}(0,t)=\mathcal{O}(t^{-\gamma}),\;\;\;\gamma>1/2,$$
then
\begin{equation}\label{j31}
\int_0^\infty|q_{xt}(0,t)|^2dt<\infty.
\end{equation}
\end{theorem}
\begin{proof}
Let $$v:=q_t,$$ then the NLS equation \eqref{nls} gives by taking
the derivative in $t$
\begin{equation}\label{34*n}
{\rm i}v_t+v_{xx}-4|q|^2v-2 q^2\bar{v}=0.
\end{equation}
 We
multiply the above with $\bar{v_x}$ and integrate in space, to
obtain after taking real parts
\begin{equation*}
\begin{split}
&\frac{{\rm i}}{2}\frac{d}{dt}(v(\cdot,t),v_x(\cdot,t))+\frac{\rm
i}{2}v(0,t)\bar{v_t}(0,t)-\frac{1}{2}|v_x(0,t)|^2\\
&-4{\rm Re}(|q(\cdot,t)|^2v(\cdot,t),v_x(\cdot,t))-2
{\rm Re}(q^2(\cdot,t)\bar{v}(\cdot,t),v_x(\cdot,t))=0.
\end{split}
\end{equation*}
We integrate the above in time in $(0,t)$ and arrive at
\begin{equation}\label{j322}
\begin{split}
\int_0^t|v_x(0,r)|^2dr=&{\rm i}(v(\cdot,r),v_x(\cdot,r))+{\rm
i}\int_0^tv(0,r)\bar{v_r}(0,r)dr\\
&-8\int_0^t{\rm
Re}(|q(\cdot,r)|^2v(\cdot,r),v_x(\cdot,r))dr- 4\int_0^t{\rm
Re}(q^2(\cdot,r)\bar{v}(\cdot,r),v_x(\cdot,r))dr.
\end{split}
\end{equation}
But it holds that for any $x$
\begin{equation*}
\begin{split}
|q(x,r)|^2\leq &\|q(\cdot,r)\|\|q_x(\cdot,r)\|\\
\leq &
c\Big{(}\int_r^\infty|q(0,s)|^2ds\Big{)}^{1/4}\Big{(}\int_r^\infty|q_s(0,s)|^2ds\Big{)}^{1/4}=:G(r),
\end{split}
\end{equation*}
where we used \eqref{j28} and \eqref{j29}. So, we get
\begin{equation}\label{j33}
\displaystyle{\max_{x\in(0,\infty)}}|q(x,r)|^2\leq G(r).
\end{equation}
Using \eqref{j33} in \eqref{j322}, we have
\begin{equation}\label{j34}
\begin{split}
\int_0^t|v_x(0,r)|^2dr=&\int_0^t|q_{tx}(0,r)|^2dr\\
\leq &c\|v(\cdot,t)\|\|v_x(\cdot,t)\|+c\int_0^t|v(0,r)||v_r(0,r)|dr\\
&+c\int_0^tG(r)\|v(\cdot,r)\|\|v_x(\cdot,r)\|dr\\
\leq
&c\|v(\cdot,t)\|^2+\hat{c_0}\|v_x(\cdot,t)\|^2+c\int_0^t|v(0,r)||v_r(0,r)|dr\\
&+c\displaystyle{\max_{r\geq
0}}[\|v(\cdot,r)\|^2+\hat{c_1}\|v_x(\cdot,r)\|^2]\int_0^tG(r)dr,
\end{split}
\end{equation}
where  $\hat{c_0}$ and $\hat{c_1}$ can be taken as small as we want (we will specify how small later).

We shall estimate $\|v(\cdot,r)\|$, $\|v_x(\cdot,r)\|$. We
multiply \eqref{34*n} by $\bar{v}$ and integrate in space to
obtain
\begin{equation}\label{j35}
\begin{split}
\frac{d}{dt}\|v(\cdot,r)\|^2=&2{\rm
Im}\{v_x(0,r)\bar{v}(0,r)\}+4{\rm
Im}(q^2(\cdot,r)\bar{v}(\cdot,r),v(\cdot,r))\\
\leq &c|v_x(0,r)||v(0,r)|+cG(r)\|v(\cdot,r)\|^2.
\end{split}
\end{equation}
But we observe that for $t\leq a_1$, $a_1$ large, then
$$\int_0^tG(r)dr\leq\int_0^{a_1}G(r)dr\leq a_1\displaystyle{\max_{r\in(0,a_1)}}G(r)=a_1G(0)\leq c,$$
while for any $t>a_1$
\begin{equation*}
\begin{split}
\int_0^tG(r)dr=\int_0^{a_1}G(r)dr+\int_{a_1}^tG(r)dr\leq &
c+c\int_{a_1}^t\Big{[}\Big{(}\int_r^\infty|q(0,s)|^2ds\Big{)}^{1/4}
\Big{(}\int_r^\infty|q_s(0,s)|^2ds\Big{)}^{1/4}\Big{]}dr\\
\leq &c+ct^{\frac{-\alpha-\beta+3}{2}}\leq c,
\end{split}
\end{equation*}
where we used that $\alpha+\beta>3$. Here, we used also that for
$t>a_1$ then $q(0,t)\sim \mathcal{O}(t^{-\alpha})$ and
$q_t(0,t)\sim \mathcal{O}(t^{-\beta})$. Thus, for any $t\geq 0$
we obtain $$e^{\int_0^tG(r)dr}\leq c,$$ uniformly in $t$, and as
a result, Gronwall's lemma, applied to \eqref{j35}, gives
\begin{equation}\label{j36}
\begin{split}
\|v(\cdot,t)\|^2 \leq c\int_0^t|v_x(0,r)||v(0,r)|dr.
\end{split}
\end{equation}
since $v(x,0)=0$. So, we have for $c_0>0$ as small we want
\begin{equation}\label{j36*}
\begin{split}
\displaystyle{\max_{t\in[0,\infty)}}\|v(\cdot,t)\|^2 \leq
c_0\int_0^\infty|v_x(0,r)|^2dr+c\int_0^\infty|q_r(0,r)|^2dr.
\end{split}
\end{equation}
In addition, \eqref{34*n} multiplied by $\bar{v_t}$ and integrated
in space gives
\begin{equation}\label{j37}
\begin{split}
&-\frac{1}{2}\frac{d}{dt}\|v_x(\cdot,r)\|^2-{\rm Re}\{v_x(0,r)\bar{v_r}(0,r)\}\\
&-4{\rm
Re}(|q(\cdot,r)|^2v(\cdot,r),v_t(\cdot,r))-2{\rm
Re}(q^2(\cdot,r)\bar{v}(\cdot,r),v_t(\cdot,r))=0.
\end{split}
\end{equation}
We use now, cf. \cite{FIS}, that
$$2{\rm Re}(q^2\bar{v},v_t)=\frac{d}{dt}{\rm
Re}(q^2,v^2)+\mathcal{A},\;\;\;\;4{\rm
Re}(|q|^2v,v_t)=2\frac{d}{dt}{\rm Re}(|q|^2,|v|^2)+\mathcal{B},$$
for $$|\mathcal{A}|+|\mathcal{B}|\leq
c\int_0^\infty|q||q_t||v|^2dx=c\int_0^\infty|q||v|^3dx.$$ So, we
get
\begin{equation}\label{j38}
\begin{split}
|\mathcal{A}+\mathcal{B}|\leq &c\int_0^\infty|q||v||v|^2dx\leq
c\int_0^\infty|q||v|(\|v\|\|v_x\|)dx\leq c\|q\|\|v\|^2\|v_x\|\\
\leq &c\|q\|\|v\|^4+c\|q\|\|v_x\|^2.
\end{split}
\end{equation}
Hence, \eqref{j37} by using \eqref{j38}, yields
\begin{equation}\label{j39}
\begin{split}
&\frac{d}{dt}\|v_x(\cdot,r)\|^2+2\frac{d}{dt}{\rm
Re}(q^2(\cdot,r),v^2(\cdot,r))+4\frac{d}{dt}{\rm
Re}(|q(\cdot,r)|^2,|v(\cdot,r)|^2)\\
\leq &
c|v_x(0,r)||v_r(0,r)|+c\|q(\cdot,r)\|\|v(\cdot,r)\|^4+c\|q(\cdot,r)\|\|v_x(\cdot,r)\|^2
\end{split}
\end{equation}
Relation \eqref{j28} yields
$$\int_0^t\|q(\cdot,r)\|dr\leq \int_0^{a_1}\|q(\cdot,r)\|dr\leq
c,$$ for $t\leq a_1$, while for $t>a_1$
\begin{equation*}
\begin{split}
\int_0^t\|q(\cdot,r)\|dr\leq &
c\int_0^t\Big{(}\int_s^\infty|q(0,r)|^2dr\Big{)}^{1/4}ds \leq
c+c\int_{a_1}^t\Big{(}\int_s^\infty|q(0,r)|^2dr\Big{)}^{1/4}ds\\
\leq& c+ct^{\frac{-2\alpha+1}{4}+1}\leq c,
\end{split}
\end{equation*}
since $\alpha>5/2$. Thus, for any $t\geq 0$, Gromwall's lemma in
\eqref{j39} gives (since $v_x(x,0)=q_{tx}(x,0)=0$ and $v(x,0)=0$)
\begin{equation}\label{j40}
\begin{split}
\|v_x(\cdot,t)\|^2\leq&
c\|q^2(\cdot,t)\|\|v^2(\cdot,t)\|+c\int_0^t\Big{[}\int_0^\infty |q(x,r)|^2|v(x,r)|^2dx\Big{]}dr\\
&+c\int_0^t|v_x(0,r)||v_r(0,r)|dr+\int_0^tG_1(r)\|v(\cdot,r)\|^4dr\\
\leq
&c\|q^2(\cdot,t)\|\|v^2(\cdot,t)\|+c\int_0^t\displaystyle{\max_{x\geq
0}}(|q(x,r)|^2)\|v(\cdot,r)\|^2dr\\
&+c\int_0^t|v_x(0,r)||v_r(0,r)|dr+c\int_0^tG_1(r)\|v(\cdot,r)\|^4dr\\
\leq
&c\|q^2(\cdot,t)\|\|v^2(\cdot,t)\|+c\int_0^t\displaystyle{\max_{x\geq
0}}(|q(x,r)|^2)1dr+c\int_0^t\displaystyle{\max_{x\geq
0}}(|q(x,r)|^2)\|v(\cdot,r)\|^4dr\\
&+c\int_0^t|v_x(0,r)||v_r(0,r)|dr+c\int_0^tG_1(r)\|v(\cdot,r)\|^4dr\\
\leq
&c\|q^2(\cdot,t)\|\|v^2(\cdot,t)\|+c\\
&+c\int_0^t|v_x(0,r)||v_r(0,r)|dr+c\int_0^t[G(r)+G_1(r)]\|v(\cdot,r)\|^4dr
\end{split}
\end{equation}
since $$\|q(\cdot,s)\|\leq
c\Big{(}\int_s^\infty|q(0,r)|^2dr\Big{)}^{1/4}=:G_1(s),$$ while as
before, for the same definition of $G$ (see \eqref{j33}),
\begin{equation*}
\begin{split}
|q(x,r)|^2\leq &\|q(\cdot,r)\|\|q_x(\cdot,r)\|\\
\leq &
c\Big{(}\int_r^\infty|q(0,s)|^2ds\Big{)}^{1/4}\Big{(}\int_r^\infty|q_s(0,s)|^2ds\Big{)}^{1/4}=:G(r),
\end{split}
\end{equation*}
and thus, $$\int_0^t\displaystyle{\max_{x\geq
0}}(|q(x,r)|^2)dr\leq\int_0^tG(r)dr\leq c.$$

Furthermore, we have
\begin{equation}\label{j41}
\begin{split}
\|v^2\|\leq c\|v\|^{3/2}\|v_x\|^{1/2}\leq c_1\|v_x\|+c\|v\|^{3}
\leq c_1\|v_x\|^2+c\|v\|^4+c,
\end{split}
\end{equation} where
 $c_1>0$ can be chosen as small as we want. Also, the same argument together
with \eqref{j28}, \eqref{j29}, gives
\begin{equation}\label{j42}
\begin{split}
\|q^2(\cdot,r)\|\leq c\|q_x(\cdot,r)\|^2+c\|q(\cdot,r)\|^4+c\leq
c,
\end{split}
\end{equation}
uniformly for any $r$. So, using \eqref{j41} and \eqref{j42} in
\eqref{j40}, we obtain
\begin{equation}\label{j43}
\begin{split}
\|v_x(\cdot,t)\|^2 \leq
&c(c_1\|v_x(\cdot,t)\|^2+c\|v(\cdot,t)\|^4+c)+c+c\int_0^t[G(r)+G_1(r)]\|v(\cdot,r)\|^4dr\\
&+c\int_0^t|v_x(0,r)|^2dr+c\int_0^t|v_r(0,r)|^2dr\\
\leq
&cc_1\|v_x(\cdot,t)\|^2+c\|v(\cdot,t)\|^4+c\int_0^t[G(r)+G_1(r)]\|v(\cdot,r)\|^4dr\\
&+c\int_0^t|v_x(0,r)|^2dr+c\int_0^t|v_r(0,r)|^2dr+c.
\end{split}
\end{equation}
Furthermore, by \eqref{j36} we arrive at
\begin{equation*}
\begin{split}
\|v(\cdot,t)\|^4 \leq &
c\Big{(}\int_0^t|v_x(0,r)||v(0,r)|dr\Big{)}^2\leq
c\Big{(}\Big{(}\int_0^t|v_x(0,r)|^2dr\Big{)}^{1/2}\Big{(}\int_0^t|v(0,r)|^2dr\Big{)}^{1/2}\Big{)}^2\\
\leq
&c\Big{(}\int_0^t|v_x(0,r)|^2dr\Big{)}\Big{(}\int_0^t|v(0,r)|^2dr\Big{)}.
\end{split}
\end{equation*}
Thus, we have
\begin{equation}\label{j44}
\begin{split}
\|v(\cdot,t)\|^4 \leq &c\int_0^t|v_x(0,r)|^2dr,
\end{split}
\end{equation}
while
\begin{equation}\label{j45}
\begin{split}
\int_0^t[G(r)+G_1(r)]\|v(\cdot,r)\|^4 dr\leq
&\displaystyle{\max_{t\geq 0}}\|v(\cdot,r)\|^4
\int_0^t[G(r)+G_1(r)]dr\\
\leq &c\int_0^\infty|v_x(0,r)|^2dr\int_0^t[G(r)+G_1(r)]dr\leq
c\int_0^\infty|v_x(0,r)|^2dr,
\end{split}
\end{equation}
since $\alpha+\beta>3$ and $\alpha>5/2$ (note that
$$\int_0^tG_1(r)dr\leq c+c\int_{a_1}^t
r^{\frac{-2\alpha+1}{4}}dr\leq c$$ for $\alpha>5/2$). Therefore,
using \eqref{j44} and \eqref{j45} in \eqref{j43}, we obtain
\begin{equation*}
\|v_x(\cdot,t)\|^2 \leq
cc_1\|v_x(\cdot,t)\|^2+c\int_0^\infty|v_x(0,r)|^2dr+c\int_0^t|v_r(0,r)|^2dr+c,
\end{equation*}
and so,  if $c_1$ is chosen small enough,
\begin{equation}\label{j46}
\|v_x(\cdot,t)\|^2 \leq
c\int_0^\infty|v_x(0,r)|^2dr+c\int_0^t|v_r(0,r)|^2dr+c,
\end{equation}

Observe now that since $\int_0^\infty G(r)dr<\infty$, then
\eqref{j34} gives
\begin{equation*}
\begin{split}
\int_0^t|v_x(0,r)|^2dr \leq
&c\|v(\cdot,r)\|^2+\hat{c_0}\|v_x(\cdot,r)\|^2+c\int_0^t|v(0,r)||v_r(0,r)|dr\\
&+c\displaystyle{\max_{r\geq
0}}[\|v(\cdot,r)\|^2+\hat{c_1}\|v_x(\cdot,r)\|^2].
\end{split}
\end{equation*}
and thus, relations \eqref{j36*} and \eqref{j46} yield
\begin{equation*}
\begin{split}
\int_0^t|v_x(0,r)|^2dr \leq
&c(c_0\int_0^\infty|v_x(0,r)|^2dr+c\int_0^\infty|q_r(0,r)|^2dr)\\
&+c(\hat{c_0}\int_0^\infty|v_x(0,r)|^2dr+\hat{c_1}\int_0^\infty|v_x(0,r)|^2dr+c\int_0^\infty|v_r(0,r)|^2dr+c)\\
&+c\int_0^\infty|v(0,r)||v_r(0,r)|dr.
\end{split}
\end{equation*}
Thus, if $c_0$, $c_2$ and $\hat{c_0}$, $\hat{c_1}$ are chosen appropriately small, we arrive at
\begin{equation*}
\begin{split}
\int_0^\infty|v_x(0,r)|^2dr \leq
c\int_0^\infty|q_r(0,r)|^2dr+c\int_0^\infty|q_{rr}(0,r)|^2dr+c\leq
c,
\end{split}
\end{equation*}
since $\beta>1/2$ and $\gamma>1/2$.
\end{proof}
\begin{remark}\label{remn2}
For the previous proof, the crucial argument was the application
of Gronwall's lemma for $t\in \mathbb{R}^+$. This could in general give
exponentially growing coefficients in time; in our case
the proven decay of solution for large times leads to the validity
of the estimates \eqref{j28} and \eqref{j29} which results to
uniformly bounded coefficients.
\end{remark}

Next we prove the following general theorem that connects the
polynomial decay of Dirichlet data to an analogous resulting decay
of the Neumann one.
\begin{theorem}\label{thm4}
Let $q$ be the
unique global classical solution $q\in C^1(L^2)\cap C^0(H^2)$ of
the problem \eqref{nls}-\eqref{ic}, with  $Q\in
C^2$ and $Q(0)=0$, with
$\lambda=1$, under the assumption
\eqref{ic0}. If as $t\rightarrow \infty$
$$q(0,t)=\mathcal{O}(t^{-\alpha}),\;\;\;\;q_t(0,t)=\mathcal{O}(t^{-\beta}),
\;\;\;\;q_{tt}(0,t)=\mathcal{O}(t^{-\gamma})$$ where
$$\alpha>5/2,\;\;\;\;\beta>5/2,\;\;\;\;\gamma>1/2,$$
then there exists $c>0$ independent of $t$ such that for any $t$
 large
\begin{equation}\label{j47}
|q_{x}(0,t)|\leq c t^{-\delta},
\end{equation}
for
$\delta=\min\Big{\{}\frac{\alpha+\beta-1}{2},\;\;\frac{4\alpha-1}{4}\Big{\}}.$
\end{theorem}
\begin{proof}

Consider $t\geq 0$ large and $s$ large such that $s<t$, then by
using first Theorem \ref{thm3} and then Proposition \ref{prop2},
we have
\begin{equation}\label{j48}
\begin{split}
|q_x(0,t)|\leq &
c\Big{(}\int_s^\infty|q_x(0,r)|^2dr\Big{)}^{1/4}\Big{(}\int_s^\infty|q_{xr}(0,r)|^2dr\Big{)}^{1/4}\\
\leq&c\Big{(}\int_s^\infty|q_x(0,r)|^2dr\Big{)}^{1/4}\Big{(}\int_0^\infty|q_{xr}(0,r)|^2dr\Big{)}^{1/4}\\
\leq &c \Big{(}\int_s^\infty|q_x(0,r)|^2dr\Big{)}^{1/4}\\
\leq & c\Big{[}\Big{(}\int_s^\infty|q(0,r)|^2dr\Big{)}^{1/4}\Big{(}\int_s^\infty|q_t(0,r)|^2dr\Big{)}^{1/4}\\
&+\Big{(}\int_s^\infty|q(0,r)|^2dr\Big{)}^{1/2}\Big{(}\int_s^\infty|q_t(0,r)|^2dr\Big{)}^{1/2}
+\int_s^\infty|q(0,r)|^4dr\Big{]}^{1/4}\\
\leq & c[s^{\frac{-2\alpha+1}{4}}s^{\frac{-2\beta+1}{4}}
+s^{\frac{-2\alpha+1}{2}}s^{\frac{-2\beta+1}{2}}
+s^{\frac{-4\alpha+1}{4}}]\\
\leq&[s^{\frac{-\alpha-\beta+1}{2}} +s^{-\alpha-\beta+1}
+s^{\frac{-4\alpha+1}{4}}]\\
\leq
&[s^{\frac{-\alpha-\beta+1}{2}}+s^{\frac{-4\alpha+1}{4}}]\leq c
s^{-\delta},
\end{split}
\end{equation}
for
$$\delta=\min\Big{\{}\frac{\alpha+\beta-1}{2},\;\;\frac{4\alpha-1}{4}\Big{\}}.$$
So, we obtain
\begin{equation*}
\displaystyle{\lim_{t\rightarrow
s^{-}}}|q_{x}(0,t)|=|q_x(0,s)|\leq c s^{-\delta},
\end{equation*}
which gives the result.
\end{proof}

In particular,  if the Dirichlet data $Q$ belong in the Schwartz class, then the
Neumann values  $q_x(0,t)$ also belong in the Schwartz class.

An immediate corollary  is the following.
\begin{theorem}\label{thm5}
Let $q$ be the
unique global classical solution $q\in C^1(L^2)\cap C^0(H^2)$ of
the problem \eqref{nls}-\eqref{ic}, with  $Q\in
C^2$ and $Q(0)=0$, with
$\lambda=1$, under the assumption
\eqref{ic0}. If as $t\rightarrow\infty$
$$q(0,t)=\mathcal{O}(t^{-5/2-\eps}),\;\;\;\;q_t(0,t)=\mathcal{O}(t^{-5/2-\eps}),
\;\;\;\;q_{tt}(0,t)=\mathcal{O}(t^{-1/2-\eps}),$$ for some small
$\eps>0$, then there exists $c>0$ independent of $t$ such that for
any $t$
 large
\begin{equation}\label{j48}
|q_{x}(0,t)|\leq c t^{-2-\eps}.
\end{equation}
In particular
\begin{equation}\label{j49}
\int_0^\infty|q_{x}(0,t)|dt<\infty.
\end{equation}
\end{theorem}
As a result, the assumptions of the unified method of scattering and inverse scattering of Fokas et al.
are now rigorously justified for the model case of defocusing NLS.

\begin{remark}
Unlike KdV and the associated Schr\"odinger operator,
scattering and inverse scattering for the NLS and the Dirac operator
has not been so thoroughly investigated in the literature.
In particular, the question of the optimal data for the
inverse scattering method to work has never been fully considered.

Our Theorems 3.9 (and 4.5 in the next section)  are motivated by the analysis in  \cite{BFS}
which implies that  an $L^1$ assumption  for  $Q, Q_t$ and $q_x(0,t)$ is enough 
for the inverse scattering method to solve the initial value problem.  
Indeed, in  \cite{BFS} an explicit formula for the eigenfunction for the second Lax operator (the t-problem)
is provided. ( \cite{BFS} does not consider infinite times, but a similar formula is possible
for our purposes.) This formula involves the solution of  a Goursat type problem for a system of PDEs.
An inspection of the formula and the system reveals \cite{S} that indeed if
$Q(t), Q_t(t)$ and $q_x(0,t)$ are  $L^1(0, \infty)$ the eigenfunction for the second Lax operator and hence the
scattering construction and the resulting Riemann-Hilbert problem are possible.
 
Of course, our Theorems  3.8 (and 4.4) are flexible enough to provide  a sufficient (and reasonably weak) condition
on the Dirichlet data that guarantees any required decay condition on the Neumann data.

An alternative approach appears in the Appendix. It is more general, but it  only bounds the
$L^1$-norm  for $q_x(0,t)$ (and for $tq_x(0,t)$).

\end{remark}

\section{Focusing NLS}\label{foc}
In this section we have to make the assumption that $q(x,t) \to 0$ as $t \to \infty$ pointwise.
While this follows easily in the case of defocusing NLS from the 4-norm estimate, it is not clear how to prove it
in the case of focusing NLS.

The following proposition holds true.
\begin{prop}\label{prop1l}
Let $q$ be the
unique global classical solution $q\in C^1(L^2)\cap C^0(H^2)$ of
the problem \eqref{nls}-\eqref{ic}, with  $Q\in
C^2$ and $Q(0)=0$, with
$\lambda=-1$, under the assumption
 \eqref{ic0}. Also let
$$\int_0^\infty|q(0,t)|^2dt,$$ be sufficiently small. Then it holds
that
\begin{equation}\label{nnn}
\|q(\cdot,t)\|^2\leq
c\Big{(}\int_t^\infty|q(0,r)|^2dr\Big{)}^{1/2},
\end{equation}
and
\begin{equation}\label{j29l}
\|q_x(\cdot,t)\|^2\leq
c\Big{(}\int_t^\infty|q_t(0,r)|^2dr\Big{)}^{1/2}+c\Big{(}\int_t^\infty|q(0,r)|^2dr\Big{)}^{3/2}.
\end{equation}
\end{prop}
\begin{proof}
The first inequality is obvious since its proof is independent of
the sign of the nonlinearity in the NLS equation, cf. the proof of \eqref{j28}.

Relation, \eqref{j0} when integrated in $(t,\infty)$ gives
\begin{equation*}
\begin{split}
\|q_x(\cdot,t)\|^2=2{\rm Re}\int_t^\infty
q_x(0,r)\bar{q_t}(0,r)dr+\|q(\cdot,t)\|_4^4,
\end{split}
\end{equation*}
Thus, using
$$\|q\|_4^4\leq \|q\|^{3}\|q_x\|\leq
\tilde{c}\|q_x\|^2+c\|q\|^6,$$ for $\tilde{c}>0$ as small as we need
and \eqref{nnn}, we obtain \eqref{j29l}. The smallness of
$\int_0^\infty|q(0,t)|^2dt,$ is needed so that Theorem \ref{thm1}
holds true and thus $\int_0^\infty |q_x(0,r)|^2dr$ is bounded.
\end{proof}
Furthermore, we proceed by proving the next estimate.
\begin{prop}\label{prop2l}
Under the assumptions of Proposition \ref{prop1l} it holds that
there exists $c>0$ independent of $t$ such that for any $t\geq 0$
\begin{equation}\label{j30}
\begin{split}
\int_t^\infty |q_x(0,r)|^2dr\leq &
c\Big{(}\int_t^\infty|q(0,r)|^2dr\Big{)}^{1/4}\Big{(}\int_t^\infty|q_t(0,r)|^2dr\Big{)}^{1/4}\\
&+c\Big{(}\int_t^\infty|q(0,r)|^2dr\Big{)}
\\&+c\Big{(}\int_t^\infty|q(0,r)|^2dr\Big{)}^{1/2}\Big{(}\int_t^\infty|q_t(0,r)|^2dr\Big{)}^{1/2}.
\end{split}
\end{equation}
\end{prop}
\begin{proof}
We integrate \eqref{ja1} in time in $(t,\infty)$ and use that
$q(x,\infty)=0$ for any $x\in (0,\infty)$ to obtain the result.
\end{proof}

Now, we can establish the following theorem.

\begin{theorem}\label{thm3l}
Under the assumptions of Proposition \ref{prop1l}, if as
$t\rightarrow \infty$
$$q(0,t)=\mathcal{O}(t^{-\alpha}),\;\;\;\;q_t(0,t)=\mathcal{O}(t^{-\beta}),
\;\;\;\;q_{tt}(0,t)=\mathcal{O}(t^{-\gamma})$$ where
$$\alpha>5/2,\;\;\;\;\beta>1/2,\;\;\;\;\gamma>1/2,$$
then
\begin{equation}\label{j31}
\int_0^\infty|q_{xt}(0,t)|^2dt<\infty.
\end{equation}
\end{theorem}
\begin{proof}
Let $$v:=q_t,$$ then the NLS equation \eqref{nls} gives by taking
the derivative in $t$
\begin{equation}\label{34*}
{\rm i}v_t+v_{xx}+4|q|^2v+2q^2\bar{v}=0.
\end{equation}
 We
multiply the above with $\bar{v_x}$ and integrate in space, to
obtain
\begin{equation*}
\begin{split}
&\frac{{\rm i}}{2}\frac{d}{dt}(v(\cdot,t),v_x(\cdot,t))+\frac{\rm
i}{2}v(0,t)\bar{v_t}(0,t)-\frac{1}{2}|v_x(0,t)|^2\\
&+4{\rm Re}(|q(\cdot,t)|^2v(\cdot,t),v_x(\cdot,t))+2
{\rm Re}(q^2(\cdot,t)\bar{v}(\cdot,t),v_x(\cdot,t))=0.
\end{split}
\end{equation*}
We integrate the above in time in $(0,t)$ and get
\begin{equation}\label{j32}
\begin{split}
\int_0^t|v_x(0,r)|^2dr=&{\rm i}(v(\cdot,r),v_x(\cdot,r))+{\rm
i}\int_0^tv(0,r)\bar{v_r}(0,r)dr\\
&+8\int_0^t{\rm
Re}(|q(\cdot,r)|^2v(\cdot,r),v_x(\cdot,r))dr+ 4\int_0^t{\rm
Re}(q^2(\cdot,r)\bar{v}(\cdot,r),v_x(\cdot,r))dr.
\end{split}
\end{equation}
But it holds that for any $x$
\begin{equation*}
\begin{split}
|q(x,r)|^2\leq &\|q(\cdot,r)\|\|q_x(\cdot,r)\|\\
\leq &
c\Big{(}\int_r^\infty|q(0,s)|^2ds\Big{)}^{1/4}\Big{(}\int_r^\infty|q_s(0,s)|^2ds\Big{)}^{1/4}\\
&+c\Big{(}\int_r^\infty|q(0,s)|^2ds\Big{)}=:\tilde{G}(r).
\end{split}
\end{equation*}
The rest of the proof is identical to that of Theorem \ref{thm3l}
with $\tilde{G}$ in place of $G$. Note that since
$\alpha>5/2,\;\;\beta>1/2$, then we have
$$\int_0^\infty \tilde{G}(r)dr<\infty.$$
\end{proof}
Now, we have a theorem for the polynomial
decay of the Neumann data  in the case of focusing NLS.
\begin{theorem}\label{thm4l}
Let $q$ be the
unique global classical solution $q\in C^1(L^2)\cap C^0(H^2)$ of
the problem \eqref{nls}-\eqref{ic}, with  $Q\in
C^2$ and $Q(0)=0$, with
$\lambda=-1$, under the assumption \eqref{ic0}. Also, let
$$\int_0^\infty|q(0,t)|^2dt,$$
be sufficiently small. If as $t\rightarrow\infty$
$$q(0,t)=\mathcal{O}(t^{-\alpha}),\;\;\;\;q_t(0,t)=\mathcal{O}(t^{-\beta}),
\;\;\;\;q_{tt}(0,t)=\mathcal{O}(t^{-\gamma}),$$ for
$$\alpha>5/2,\;\;\;\;\beta>1/2,\;\;\;\;\gamma>1/2,$$
then there exists $c>0$ independent of $t$ such that for
any $t$
 large
\begin{equation}\label{j47}
|q_{x}(0,t)|\leq c t^{-\delta},
\end{equation}
for
$\delta=\min\Big{\{}\frac{\alpha+\beta-1}{2},\;\;2\alpha-1\Big{\}}.$
\end{theorem}
\begin{proof}
The proof is analogous to this of Theorem \ref{thm4}.
\end{proof}

Again, the following corollary is obvious.
\begin{theorem}\label{thm5}
Let $q$ be the
unique global classical solution $q\in C^1(L^2)\cap C^0(H^2)$ of
the problem \eqref{nls}-\eqref{ic}, with  $Q\in
C^2$ and $Q(0)=0$, with
$\lambda=-1$, under the assumption
\eqref{ic0}. Also, let
$$\int_0^\infty|q(0,t)|^2dt,$$
be sufficiently small.  If as $t\rightarrow\infty$
$$q(0,t)=\mathcal{O}(t^{-5/2-\eps}),\;\;\;\;q_t(0,t)=\mathcal{O}(t^{-5/2-\eps}),
\;\;\;\;q_{tt}(0,t)=\mathcal{O}(t^{-1/2-\eps}),$$ for some small
$\eps>0$, then there exists $c>0$ independent of $t$ such that for
any $t$
 large
\begin{equation}\label{j48}
|q_{x}(0,t)|\leq c t^{-2-\eps},
\end{equation}
and thus,
\begin{equation}\label{j49}
\int_0^\infty|q_{x}(0,t)|dt<\infty.
\end{equation}
\end{theorem}
As a result, the decay assumption of the unified method of scattering and inverse scattering of Fokas et al.
is now justified for the case of focusing NLS, at least under the stated assumptions.

We also note that if the Dirichlet data $Q$ belong in the Schwartz class, then the
Neumann values  $q_x(0,t)$ also belong in the Schwartz class.
\section{Conclusions}
The main result  of this paper is the following: in order to
obtain fast enough decay  for the Neumann data $q_x(0,t)$
so that the Fokas method is applicable, it is sufficient to impose
a certain (mild) polynomial decay on the Dirichlet data $q(0,t)$  and its derivatives $q_t(0,t)$ and
$q_{tt}(0,t)$.

Of course the  decay assumed here for the Dirichlet data, although
sufficient, may not be the optimal one.

For the focusing NLS the answer is only provided under the assumption
that the solution tends to $0$ as $t\rightarrow\infty$ and a smallness
assumption for $\int_0^\infty|q(0,t)|^2dt$. The
general problem remains open for this case, as for the moment
we do not posses a proof  of this decay.

\appendix
\section{On the $L^1$ estimate}

We present for completeness an alternative and more general
approach for deriving decay estimates and the $L^1$ estimate for the
Neumann data.

Using that for all positive $x$,
\begin{equation}
\displaystyle{\lim_{t\rightarrow\infty}}q(x,t)=0,\;\;\;\;
\end{equation}
(proved for the defocusing NLS, cf. \eqref{apo2}, assuming as we did in section 4 that it
is true for the focusing NLS as well) we can prove a
general $L^1(0,\infty)$ bound for $q_x(0,t)$, for both cases, as
follows.

Consider the relation \eqref{ja1} (true for both the defocusing and
focusing NLS)
\begin{equation*}\label{j1}
|q_x(0,t)|^2={\rm i}\frac{d}{dt}(q,q_x)+{\rm
i}q(0,t)\bar{q_t}(0,t)+\lambda|q(0,t)|^4.
\end{equation*}
Take $p>1$ and multiply the above with $t^p$ to obtain
\begin{equation}\label{ap1}
t^p|q_x(0,t)|^2={\rm i}t^p\frac{d}{dt}(q,q_x)+{\rm
i}t^pq(0,t)\bar{q_t}(0,t)+\lambda t^p|q(0,t)|^4.
\end{equation}
Integration of \eqref{ap1} in time gives
\begin{equation*}
\begin{split}
\int_0^\infty t^p|q_x(0,t)|^2dt=&{\rm i}\int_0^\infty
t^p\frac{d}{dt}(q(\cdot,t),q_x(\cdot,t))dt+{\rm i}\int_0^\infty
t^pq(0,t)\bar{q_t}(0,t)dt+\lambda \int_0^\infty t^p|q(0,t)|^4dt\\
= &{\rm i} [t^p(q(\cdot,t),q_x(\cdot,t))]_0^\infty-{\rm i}\int_0^\infty
pt^{p-1}(q(\cdot,t),q_x(\cdot,t))dt\\
&+{\rm i}\int_0^\infty t^pq(0,t)\bar{q_t}(0,t)dt+\lambda
\int_0^\infty t^p|q(0,t)|^4dt,
\end{split}
\end{equation*}
and thus
\begin{equation}\label{ap2}
\begin{split}
\int_0^\infty t^p|q_x(0,t)|^2dt\leq&
\displaystyle{\lim_{t\rightarrow\infty}}\Big{(}t^p\|q(\cdot,t)\|\|q_x(\cdot,t)\|\Big{)}+c\int_0^\infty
t^{p-1}\|q(\cdot,t)\|\|q_x(\cdot,t)\|dt\\
&+c\int_0^\infty t^p|q(0,t)||q_t(0,t)|dt+\lambda \int_0^\infty
t^p|q(0,t)|^4dt.
\end{split}
\end{equation}
Since under the assumptions of
Theorem \ref{thm2} it holds
(cf. the proof of \eqref{j28}) that there exists $c>0$ independent
of $t$ such that for any $t\geq 0$
$$
\|q(\cdot,t)\|^2\leq
c\Big{(}\int_t^\infty|q(0,r)|^2dr\Big{)}^{1/2},$$ while $\|q_x\|$
is bounded uniformly in $t$, we obtain by \eqref{ap2}

\begin{equation}\label{ap22}
\begin{split}
\int_0^\infty t^p|q_x(0,t)|^2dt\leq&
c\displaystyle{\lim_{t\rightarrow\infty}}\Big{(}t^p\Big{(}\int_t^\infty|q(0,r)|^2dr\Big{)}^{1/4}\Big{)}
\\&+c\int_0^\infty
t^{p-1}\Big{(}\int_t^\infty|q(0,r)|^2dr\Big{)}^{1/4}dt\\
&+c\int_0^\infty t^p|q(0,t)||q_t(0,t)|dt+\lambda \int_0^\infty
t^p|q(0,t)|^4dt.
\end{split}
\end{equation}

Furthermore,
$$t^p\|q(\cdot,t)\|\|q_x(\cdot,t)\|\leq
ct^p\Big{(}\int_t^\infty|q(0,r)|^2dr\Big{)}^{1/4}\rightarrow 0$$
if $q(0,t)$ has a sufficiently fast  (polynomial) decay as
$t\rightarrow\infty$. Also, we have
$$\int_0^\infty t^{p-1}\|q(\cdot,t)\|\|q_x(\cdot,t)\|dt\leq
c\int_0^\infty
t^{p-1}\Big{(}\int_t^\infty|q(0,r)|^2dr\Big{)}^{1/4}dt\leq c,$$ if
again $q(0,t)$ has a sufficiently fast decay as
$t\rightarrow\infty$. The same argument of sufficiently fast
decay for $q_t(0,t)$ as $t\rightarrow\infty$, together with the
previous one,  finally gives, using \eqref{ap22}
\begin{equation}\label{ap3}
\begin{split}
\int_0^\infty t^p|q_x(0,t)|^2dt\leq c,
\end{split}
\end{equation}
as $p>1$.

Note that a weaker assumption on initial data $q(0,t)$ and
$q_t(0,t)$ resulting in  a bounded right-hand side of relation
\eqref{ap22} also gives \eqref{ap3}. This observation makes this
approach more general.

Now, we are ready to  derive the  $L^1(0,\infty)$ estimate for
$|q_x(0,t)|$. Indeed, 
\begin{equation}\label{ap4}
\begin{split}
\int_0^\infty |q_x(0,t)|dt&=\int_0^1|q_x(0,t)|dt+\int_1^\infty
t^{-\frac{1}{2}-\frac{\eps}{2}}t^{\frac{1}{2}+\frac{\eps}{2}}|q_x(0,t)|dt\\
&\leq
c\Big{(}\int_0^11^2dt\Big{)}^{1/2}\Big{(}\int_0^1|q_x(0,t)|^2dt\Big{)}^{1/2}+
\Big{(}\int_1^\infty
t^{-1-\eps}dt\Big{)}^{1/2}\Big{(}\int_1^\infty
t^{1+\eps}|q_x(0,t)|^2dt\Big{)}^{1/2}\\
&\leq c+ c\Big{(}\int_0^\infty
t^{1+\eps}|q_x(0,t)|^2dt\Big{)}^{1/2}\leq c,
\end{split}
\end{equation}
where we used \eqref{ap3} for $p:=1+\eps$.

\begin{remark}
We could, in fact, proceed to a  detailed presentation of the sufficient
order of decay of $q(0,t)$ and $q_t(0,t)$ which will guarantee a required order of decay
of $q_x(0,t)$, as we did in  Sections 3 and 4. We leave the tedious calculations to the interested reader.  
\end{remark}

\begin{remark}
It is easy to also derive the  $L^1(0,\infty)$ condition  for
$t|q_x(0,t)|$, if one wishes.
As in \eqref{ap4}, we have
\begin{equation}\label{ap4*}
\begin{split}
\int_0^\infty t|q_x(0,t)|dt&=\int_0^1t|q_x(0,t)|dt+\int_1^\infty
tt^{-\frac{1}{2}-\frac{\eps}{2}}t^{\frac{1}{2}+\frac{\eps}{2}}|q_x(0,t)|dt\\
&\leq
c\Big{(}\int_0^1t^2dt\Big{)}^{1/2}\Big{(}\int_0^1|q_x(0,t)|^2dt\Big{)}^{1/2}+
\Big{(}\int_1^\infty
t^{-1-\eps}dt\Big{)}^{1/2}\Big{(}\int_1^\infty
t^2t^{1+\eps}|q_x(0,t)|^2dt\Big{)}^{1/2}\\
&\leq c+ c\Big{(}\int_0^\infty
t^{3+\eps}|q_x(0,t)|^2dt\Big{)}^{1/2}\leq c,
\end{split}
\end{equation}
where we used \eqref{ap3} for $p:=3+\eps$.
\end{remark}

\begin{remark}
The approach of the Appendix is more general although it  only succeeds in giving the
$L^1$ condition for $q_x(0,t)$ and for $tq_x(0,t)$ (which of course is enough for our purposes). 
In Sections
\ref{def} (after Theorem \ref{thm2} and the resulting
\eqref{apo2}), and \ref{foc}, we essentially imposed the extra conditions that
the limits of certain derivatives exist  as
$t\rightarrow \infty$ (and thus, they are $0$, due for
example to the fact that $q(x,t)\rightarrow 0$ as
$t\rightarrow\infty$ uniformly for any $x$, or the $L^2(0,t)$
boundedness as $t\rightarrow\infty$). Thus, we were able to prove
stronger decay estimates for $q_x(0,t)$, which in turn  led also to the $L^1$
condition.
\end{remark}

\end{document}